\documentclass[12pt,reqno]{amsart}

\usepackage{latexsym}
\usepackage{amssymb}

\newcommand{\s}{{\mathbf s}}

\renewcommand{\epsilon}{\varepsilon}

\newcommand{\PP}{{\mathbb P}}
\newcommand{\N}{{\mathbb N}}

\newcommand{\C}{{\mathbb C}}

\newcommand{\B}{{\mathbb B}}
\newcommand{\CP}{\C\PP}

\renewcommand{\phi}{\varphi}

\newtheorem{theo}{{Theorem}}[section]
\newtheorem{cor}[theo]{{Corollary}}
\newtheorem{lem}[theo]{{Lemma}}

\title[excursion sets of random holomorphic sections on complex manifolds]{expected euler characteristic of excursion sets of random holomorphic sections on complex manifolds}

\author{Jingzhou Sun }
\address{Department of Mathematics, Johns Hopkins University, Baltimore, MD
21218, USA} \email{jzsun@math.jhu.edu}

\thanks{Research  partially supported by NSF grant
 DMS-0901333.}

\date{\today}

\begin{document}

\begin{abstract}
We prove a formula for the expected euler characteristic of excursion sets of random sections of powers of an ample bundle $(L,h)$, where $h$ is a Hermitian metric, over a K\"{a}hler manifold $(M,\omega)$. We then prove that the critical radius of the Kodaira embedding $\Phi_N:M\rightarrow \CP^n$ given by an orthonormal basis of $H^0(M,L^N)$ is bounded below when $N\rightarrow \infty$. This result also gives conditions about when the preceding formula is valid.

\end{abstract}

\maketitle

 \tableofcontents
\section{Introduction}
Let  $M$ be a K\"{a}hler manifold of dimension $m$. And let $L\rightarrow M$ be an ample line bundle  with positively curved metric $h$. Take the induced K\"{a}hler form
$\omega=\frac{i}{2}\Theta_h$ on $M$. We denote by $L^N$ the $N$th tensor power $L^{\otimes N}$ of $L$. Take the induced metric on $L^N$, by abuse of notation, also denoted by $h$. This induces a Hermitian inner product in $H^0(M,L^N)$, given by $$<\sigma_1,\sigma_2>=\frac{1}{m!}\int_Mh(\sigma_1,\sigma_2) \omega^m$$. In particular, the $L^2$ norm of a section in $s\in H^0(M,L^N)$ is given by $$|s|^2_h=\frac{1}{m!}\int_M|s(z)|^2_h\omega^m$$
 We consider random sections in the unit sphere $S_L^N$ in $H^0(M,L^N)$ with
  probability measure given by the spherical volume normalized so that
  $Vol(S_L^N)=1$.
For $s\in S_L^N$, the zero locus $Z_s=\{z\in M|s(z)=0\}$ is very well studied in \cite{sz65}\cite{sz2}\cite{sz3}\cite{sz4}. It is also interesting to understand the excursion sets $\{z\in M||s(z)|_h>u\}$. In particular, what is $E\chi(|s(z)|_h>u)=\int_{S_L^N}\chi(|s(z)|_h>u)ds$, the expected Euler characteristic of the excursion sets, and what is the probability that the excursion set is non-empty. Here and in the following we denote by $\chi(S)$ the Euler characteristic of a topological space $S$.

It turned out that it is more natural to normalize the excursion sets to be of the form $\{\frac{|s(z)|_h}{\sqrt{\Pi_N(z,z)}}>u\}$, where $\Pi_N(z,z)$, which is in general not constant but of the form $\frac{N^m}{\pi^m}(1+O(\frac{1}{N}))($\cite{ca}\cite{ti}\cite{Z}),  is the Szeg\"{o} kernel of $H^0(M,L^N)$.
By the definition of $\Pi_N(z,z)$, we always have $\frac{|s(z)|_h}{\sqrt{\Pi_N(z,z)}}\leq 1$. In fact $\sup_{|s|_h=1}|s(z)|_h^2=\Pi_N(z,z)$(\cite{bo}). Therefore when $u>1$, the excursion sets are empty.

In this paper, we will mainly prove two theorems.

The first theorem is interesting in itself. Also it shows that in order to have a nice formula for the expected Euler characteristic we do not need to make $u$ too close to $1$.
\begin{theo}\label{theo:critical}
Let $\Phi_N:M\rightarrow \CP^n$ be an embedding given by an orthonormal basis of $H^0(M,L^N)$. Let $r_N$ be the critical radius of $\Phi_N(M)$ considered as a submanifold of $\CP^n$. Then there exists a constant $\rho_0(L,h)>0$ such that $r_N>\rho_0(L,h)$ for all positive integer $N$.
\end{theo}
\smallskip
The proof of this theorem depends mainly on the approximation of the normalized Szeg\"{o} kernel defined and proved in \cite{sz65}. The idea is based on the sense that the information of the embedding  $\Phi_N:M\rightarrow \CP^n$ is totally contained in the normalized Szeg\"{o} kernel.
\bigskip

The second one is to answer the question about expected Euler characteristic of the normalized excursion set.
\begin{theo}\label{theo:1}
Let $\Phi_N:M\rightarrow \CP^n$ be an embedding given by an orthonormal basis of $H^0(M,L^N)$.
Then there exists $\rho_0>0$ independent of $N$, such that for $0\leq \rho<\rho_0$, the set $\{z\in M|\frac{|s(z)|_h}{\sqrt{\prod_N(z,z)}}>\cos\rho\}$ is either empty or contractible, therefore$$E\chi(\frac{|s(z)|_h}{\sqrt{\prod_N(z,z)}}>\cos\rho)=Prob.\{\sup_{z\in M}\frac{|s(z)|_h}{\sqrt{\prod_N(z,z)}}>\cos\rho\}$$Hence the following formula
\begin{equation}\label{formula 1}
E\chi(\frac{|s(z)|_h}{\sqrt{\prod_N(z,z)}}>\cos\rho)=\int _M
c(M)(1-Nc_1(L))\wedge(Nc_1(L)\cos^2\rho+ \sin^2\rho)^n
\end{equation}
Where $c_1(L)$ is the first Chern class of $L$ and $c(M)(1-Nc_1(L))$ is the Chern
polynomial evaluated at $1-Nc_1(L)$
\end{theo}
\smallskip

When $M$ is a Riemann surface, we have a more explicit formula

\begin{theo}\label{theo:eec}
Let $M$ be a Riemann surface. Then, with the notations above, there exists $\rho_N>0$ such that for $u>\cos\rho_N$ and a random section $s(z)\in H^0(M,L^N)$ the expected
Euler characteristic
\begin{eqnarray}E\chi(\frac{s(z)}{\sqrt{\Pi_N(z,z)}}>u)&=&(1-u^2)^{(n-1)}[N^2(\deg L)^2u^2\\&-&N\deg L(gu^2-1+u^2)+(2-2g)(1-u^2)]
\end{eqnarray}
where $n=N\deg(L)-g$ for $N\deg(L)>2g-2$
\end{theo}

When $M$ is higher dimensional, we can only get an estimation
\begin{theo}\label{theo:estimation}
With the notations above, for $m\geq 1$ and for $N$ big enough
$$E\chi(\frac{s(z)}{\sqrt{\Pi_N(z,z)}}>u)=(1+o(1))n^{m+1}(1-u^2)^{n-m}u^{2m}$$
where $$n=\dim H^0(M,L^N) -1=\frac{\int_Mc_1^m(L)}{m!}N^m+O(N^{m-1})$$
where the second equality follows from the asymptotic Riemann-Roch formula.
\end{theo}
\smallskip

Our results are complementary to results on excursion probabilities for Gaussian fields (see \cite{at}\cite{su}) where the probability of large $L^2$ norms plays a role.  Here, we consider only sections with $L^2$ norm 1.

Notice that by our estimation $E\chi(\frac{s(z)}{\sqrt{\Pi_N(z,z)}}>u)$ decays to 0 very rapidly (exponentially) as $N$ goes to $\infty$. It is helpful to compare this observation with the following theorem from \cite{sz levy}, which we state using our notations
\begin{theo}(Theorem 1.1, \cite{sz levy})Let $\nu_N$ denote the measure on $S_L^N$ induced by the metric $ds$, for any integer $k$, there exist constants $C>0$ depending on $k$, such that
$$\nu_N\{s_N\in S_L^N:\sup_{z\in M}|s_N(z)|_h>C\sqrt{\log N}\}<O(\frac{1}{N^k})$$

\end{theo}
Normalizing the above formula using $\sqrt{\Pi_N(z,z)}$, and by the estimation of $\sqrt{\Pi_N(z,z)}$, we have $$\nu_N\{s_N\in S_L^N:\sup_{z\in M}\frac{|s_N(z)|_h}{\sqrt{\Pi_N(z,z)}}>\frac{C\sqrt{\log N}}{N^{m/2}}\}<O(\frac{1}{N^k})$$
the term $\frac{C\sqrt{\log N}}{N^m}$ is very small when $N$ is big. But the estimation we made requires $u$ close to 1, so in this sense our estimation is weaker, although it is more explicit.
\smallskip

 It should be mentioned here that in the proof of theorem \ref{theo:1}, we make $\rho_0$ small enough so that the Euler characteristic of an excursion set is either 1 or 0. The author does not have an idea on the general situation when $\rho_0$ is greater than the critical radius(section \ref{sec:critical radius}). It seems that we need to understand the exterior geometry of the embedded manifolds better. It is an interesting question to ask, for example, when does an excursion set have 2 components? More generally, what is the volume of the set of sections whose excursion sets have $k$ components.
\smallskip

The paper is organized as follows: first in section \ref{sec:background} we introduce the definition of Szeg\"{o} Kernel and state several results from \cite{sz2},\cite{sz65},and \cite{Z}. In section \ref{sec:generalsettings} we prove the formula in theorem \ref{theo:1}. In section \ref{sec:critical radius} we will analyze the critical radius, by using the results stated in section \ref{sec:background} first in the case of Riemann surfaces then generalize to high dimensional case.

\textbf{Acknowledgement}:
The author would like to thank Professor Bernard Shiffman for instructive discussions, continuous guidance and consistent patience in answering all my questions and in correcting all the mistakes in this paper.  The author would also like to thank Yuan Yuan and Caleb Hussey  for helpful discussions.
\section{Szeg\"{o} Kernel}\label{sec:background}
We will follow the notations and arguments in
\cite{sz65} and \cite{sz2}

 Let $L\rightarrow M$ be a positive line bundle over a
K\"{a}hler manifold $M$. The associated principle sphere bundle is defined as follows. Let $\pi:L^*\rightarrow M$ be the dual bundle to $L$ with dual metric $h^*$. And put $X=\{v\in L^*:\parallel v\parallel_h^*=1\}$. Let $r_{\theta }x=e^{i\theta}x (x\in X)$ denote the $S^1$ action on $X$. Now any section $s\in H^0(M,L^N)$
is lifted to an equivariant function $\hat{s}$ on the circle bundle $\pi:X\rightarrow M$
with respect to $h$ by the rule $$\hat{s}(\lambda)=(\lambda^{\otimes N},s(z)), \lambda\in X_z$$
where $\lambda^{\otimes N}=\lambda\otimes\cdots \otimes\lambda$. Let $(s_j^N)\subset H^0(M,L^N)$ be an orthonormal basis.
The Szeg\"{o} kernel is given by
$$\Pi_{N}(x,y)=\sum_{i=0}^n\hat{s}_j^N(x)\hat{s}_j^N(y)\qquad (x,y\in X)$$
The normalized Szeg\"{o} kernel is defined as
$$P_N(z,w)
:=\frac{|\prod_N(z,w)|}{\prod_N(z,z)^{1/2}\prod_N(w,w)^{1/2}}$$
where
\begin{equation}
|\Pi_N(z,w)|:=|\Pi_N(x,y)|,\ \ z=\pi(x), w=\pi(y)\in M.
\end{equation}
On the diagonal we have $$\Pi_{N}(z,z)=\sum_{i=0}^n\parallel\s_j^N(z)\parallel_h^2,z\in M$$
The following theorem was proved in \cite{Z}
\begin{theo}[\cite{Z}]
Let $M $ be a compact complex manifold of dimension $m$(over $\C$) and let $(L,h)\rightarrow M$ be a positive hermitian holomorphic line bundle. Let $g$ be the K\"{a}hler metric on $M$ corresponding to the K\"{a}hler form $\omega_g:=Ric(h)$. For each $N\in \N$, $h$ induces a hermitian metric $h_N$ on $L^{\otimes N}$. Let $\{S_0^N,\cdots,S_{d_N}^N\}$ be any orthonormal basis of $H^0(M,L^{\otimes N})$, with respect to the inner product $<s_1,s_2>_{h_N}=\int_Mh_N(s_1(z),s_2(z))dV_g$. Here, $dV_g=\frac{1}{m!}\omega_g^m$ is the volume form of $g$. Then there exists a complete asymptotic expansion:
$$\sum_{i=0}^{d_N}\parallel S_i^N(z)\parallel_{h_N}^2\sim a_0N^m+a_1(z)N^{m-1}+\cdots$$
for certain smooth coefficients $a_j(z)$ with $a_0=\frac{1}{\pi^m}$. More precisely, for any $k$
$$|\sum_{i=0}^{d_N}\parallel S_i^N(z)\parallel_{h_N}^2-\sum_{j<R}a_j(x)N^{m-j}|_{C^k}\leq C_{R,k}N^{m-R} $$
\end{theo}

At a point $z_0\in M$, we choose a
neighborhood $U$ of $z_0$, a local normal coordinate chart $\rho:U,
z_0\rightarrow \C^m$, $0$ centered at $z_0$, and a preferred local
frame at $z_0$, which was defined in \cite{sz2} to be a local frame $e_L$
such that $$|e_L(z)|_h=1-1/2|\rho(z)|^2+\cdots.$$

The following theorem was proved in \cite{sz65}

\begin{theo}[\cite{sz65}, Proposition 2.7]\label{theo:sz1}
Let $(L, h)\rightarrow (M, \omega)$ be a positive Hermitian
holomorphic line bundle over a compact m-dimensional K\"{a}hler
manifold $M$. We give $H^0(M,L^N)$ the Hermitian Gaussian measure
induced by $h$ and the K\"{a}hler form $\omega=\frac{i}{2}\Theta_h$.
And let $P_N(z,w)$ be the normalized Szeg\"{o} kernel for $H^0(M,
L^N)$ and let $z_0\in M$ For $b,\epsilon>0,j\geq 0$, there is a
constant $C_j=C_j(M,\epsilon,b)$, independent of the point $z_0$,
such that

\begin{eqnarray}
P_N(z_0+\frac{u}{\sqrt{N}},z_0+\frac{v}{\sqrt{N}})=e^{-\frac{1}{2}|u-v|^2}[1+R_N(u,v)]\\|\bigtriangledown^jR_N(u,v)|\leq
C_jN^{-1/2+\epsilon } \quad for \quad|u|+|v|<b\sqrt{log N}
\end{eqnarray}

\end{theo}

As a corollary we have
\begin{theo}[\cite{sz65}, Proposition 2.8]\label{theo:sz2}
The remainder $R_N$ in the above theorem satisfies
\begin{eqnarray}
|R_N(u,v)|\leq \frac{C_2}{2}|u-v|^2N^{-\frac{1}{2}+\epsilon}, \quad
|\bigtriangledown R_N(u,v)|\leq
C_2|u-v|N^{-\frac{1}{2}+\epsilon},\\for  \quad |u|+|v|<b\sqrt{log
N}. \qquad  \qquad \qquad \qquad
\end{eqnarray}
\end{theo}

\section{Expected Euler Characteristic}\label{sec:generalsettings}
 Let $s_N^j\in
H^0(M,L^N),0\leq j \leq n$, where $n+1=\dim(H^0(M,L^N))$ be an orthonormal basis. By the Kodaira embedding theorem, for $N$ big enough this gives an embedding $\Phi_N:
M\rightarrow \CP^N$, locally given by $\Phi_N(x)=[f_N^0(x),f_N^1(x),\cdots,f_N^n(x)]$, where $s_N^j=f_N^je_L^N$ with $e_L$ a local frame of $L$. For a random section with
norm 1, $$s=\sum_{i=0}^{N}c_is_n^i, \ \ \sum_{i=0}^{N}\parallel
c_i\parallel ^2=1$$ Let $|s(z)|_h$ denote the
norm of $s$ at $z\in M$ under the metric induced by $h$. Let
$C=(c_i)$ and $f_N(z)=(f_N^i(z))$.Then $\sum_{i=0}^{N}|s_N^i(z)|_h^2=|f_N(z)|^2|e_L(z)|_h^{2N}$$$|s(z)|_h=|C\cdot
f_N(z)||e_L|^N=\frac{C\cdot
f_N(z)}{|C||f_N(z)|}\sqrt{\sum_{i=0}^{N}|s_N^i(z)|_h^2}$$ By
definition $\Pi_N(z,z)=\sum_{i=0}^{n}|s_N^i(z)|_h^2$ is just the
Szego kernel for $H^0(M,L^N)$ on the diagonal.

Therefore we have $$\frac{|s(z)|_h}{\sqrt{\Pi_N(z,z)}}=\frac{|C\cdot
\Phi_n(z)|}{|C||\Phi_n(z)|}$$
Notice that with the Fubini-Study metric on $\CP^n$, $$\cos
d_{FS}(C,\Phi_N(Z))=\frac{|C\cdot \Phi_N(Z)|}{|C||\Phi_n(Z)}|$$
So we get
$$\frac{|s(Z)|_h}{\sqrt{\Pi_N(z,z)}}=\cos
d_{FS}(C,\Phi_N(Z))$$
\begin{lem}
Let $M$ be a compact submanifold of $(A,g)$, where $(A,g)$ is a $C^{\infty}$
Remannian manifold, let $B_{\rho}(P)$denote the ball centered at $P\in
A$ with radius $\rho$, then for $\rho>0$ small enough,
$\B_{\rho}(P)\cap M$ is contractible if not empty.
\end{lem}
\begin{proof}
Consider the normal bundle of $\pi:N\rightarrow M$ in $A$ and the
exponential map $\exp:N\rightarrow A$. Since $M$ is compact, there
exists $\rho_1>0$ such that restricted to the open neighborhood
$O_M(\rho_1)=\{(p,v)|\parallel v\parallel <\rho_1\}\subset N$, the
exponential map is injective. Now we claim that any $\rho<\rho_1$
satisfies the requirement of the lemma.

Now suppose $\B_{\rho}(P)\cap M$ is not empty, then $P\in
\exp(O_M(\rho_1)) $ and $P=\exp(p,v)$ with $(p,v)\in O_M(\rho_1)$.
Consider $d(P,-)$ as a smooth function on $\B_{\rho}(P)\cap M$, then
by assumption $p$ is the only critical point of $d(P,-)$, since the geodesic that connects $p$ and a critical point of $\B_{\rho}(P)\cap M$ is orthogonal to $M$.  Let
$r=d(P,p)$, then $\B_{r}(P)\cap M=p$. So by Morse theory
$\B_{\rho}(P)\cap M$ is contractible.

\end{proof}

\begin{cor}
Let $r_N$ be the critical radius of the embedding $\Phi_N(M)\subset \CP^n$, then for $\rho<r_N$, the excursion set $\{z\in M|\frac{|s(z)|_h}{\sqrt{\Pi_N(z,z)}}>\cos\rho\}$ is either contractible or empty.
\end{cor}
By theorem \ref{theo:critical}(which will be proved in the next section), $r_N$ is bounded below by $\rho_0>0$. Therefore as a corollary, taking into account that the Fubini-Study metric on $\CP^n $ is the quotient of the "round metric" under the fibration $S^1\rightarrow S^{2n+1}\rightarrow \CP^n$ we have
$$E\chi(\frac{|s(z)|_h}{\sqrt{\Pi_N(z,z)}}>\cos\rho)=Prob.\{\sup_{z\in M}\frac{|s(z)|_h}{\sqrt{\prod_N(z,z)}}>\cos\rho\}=\frac{Vol
(T(\Phi_N(M),\rho))}{Vol(\CP^n)}$$ for $\rho<\rho_0$

First we calculate the volume $V (T(\Phi_N(M),\rho))$. We use
theorems and formulas from \cite{G2}(Theorem 7.20)(see also \cite{G1}).
\begin{theo}
Let $M^m$ be an embedded complex submanifold of
$(\CP^n,\omega_{FS})$, and let $N$ be the normal bundle of $M$ in $\CP^n$ suppose that $\exp:\{(p,v)\in N|\parallel
v\parallel <r\}\rightarrow T(M,r)$ is a diffeomorphism. Then
$$V_M(r)=\frac{1}{n!}\int _M
\prod_{a=1}^m(1-\frac{\omega_{FS}}{\pi}+x_a)\wedge(\pi\sin^2(r)+\cos^2(r)\omega_{FS})^n$$
Where $x_a$ is defined formally in factorization of the Chern
polynomial $c(M)(t)=\prod_{a=1}^m(t+x_a)$
\end{theo}
As a corollary of this theorem and by plugging in $\Phi_N^*(\omega_{FS})=N\pi c_1(L)$,and dividing by the Fubini-Study volume $\pi^n/n!$ of $CP^n$,  we get theorem  \ref{theo:1}
\bigskip

When $M$ is a Riemann surface, $m=1$, so $x_1=c_1(M)$ the first Chern
class. So
$$V_M(r)=\frac{1}{n!}\int_M(1-\frac{\omega_{FS}}{\pi}+c_1(M))\wedge
[(\pi \sin^2(r))^n+n(\pi\sin^2(r))^{n-1}\cos^2(r)\omega_{FS}]$$
therefore $$V_M(r)=\frac{1}{n!}\int_M
[(\pi\sin^2(r))^n(c_1(M)-\frac{\omega_{FS}}{\pi})+n(\pi\sin^2(r))^{n-1}\cos^2(r)\omega_{FS}]$$
We know by the Gauss-Bonnet formula $\int_Mc_1(M)=\chi(M)=2-2g$ and since $\Phi_N(N)$ is of degree $N\deg(L)$ in
$\CP^n$, $\int_M \omega_{FS}=N\deg(L)\pi$. Now we can write out the explicit
formula for $V (T(\Phi_N(M),\rho))$, that is \begin{eqnarray}V
(T(\Phi_N(M),\rho))=\frac{1}{n!}[(\pi\sin^2(\rho))^n(\chi(M)-N\deg(L))+nN\deg(L)\pi(\pi\sin^2(\rho))^{n-1}\cos^2(\rho)]\\
=\frac{\pi^n}{n!}(\sin^{2(n-1)}\rho)[N^2(\deg L)^2\cos^2\rho-N\deg L(g\cos^2\rho-\sin^2\rho)+(2-2g)\sin^2\rho]
\end{eqnarray}
where $\chi(M)=2-2g$ and by Riemann-Roch formula $n=N\deg(L)-g$ for $N\deg(L)>2g-2$

To summarize, we have the following theorem
\begin{theo}
Let $M$ be a Riemann surface. Then, with the notations above, there exists $\rho_0>0$ such that for $\rho<\rho_0 , N deg L>2g-2$ and a random section $s(z)\in H^0(M,L^N)$ the expected
Euler characteristic
\begin{eqnarray}E\chi(\frac{s(z)}{\sqrt{\Pi_N(z,z)}}>\cos \rho)\qquad \qquad \qquad\qquad\qquad\qquad\\
=(\sin^{2(n-1)}\rho)[N^2(\deg L)^2\cos^2\rho-N\deg L(g\cos^2\rho-\sin^2\rho)+(2-2g)\sin^2\rho]
\end{eqnarray}

\end{theo}
\smallskip

If we write $u=\cos \rho$ and plug in $\sin^2\rho=1-u^2$, we get theorem \ref{theo:eec}
\smallskip

Note that when $m>1$, the expansion of $E\chi(\frac{s(z)}{\sqrt{\Pi_N(z,z)}}>\cos \rho)$ is complicated and the author can not get a more intuitive formula. We can calculate the leading term to have an estimation of formula \ref{formula 1}.

Observe that the leading term in the expansion should be

$$\frac{1}{\pi^n}\int_M\left(\begin{array}{c} n\\m \end{array} \right)(\pi\sin^2\rho)^{n-m}(\cos^2\rho \omega_{FS})^m\label{leading term}$$

Let $O_n(1)$ denote the hyperplane bundle on $\CP^n$.Then $\omega_{FS}$ is a multiple of the first Chern class of $O_n(1)$ that is $\omega_{FS}=\pi c_1(O_n(1))$. Also the pull back $\Phi_N^*(c_1(O_n(1)))=N c_1(L)$. Therefore we have $$\int_M\omega_{FS}^m=\pi^mN^m\int_Mc_1^m(L)$$ which is independent of the metric on $L$.

So formula \ref{leading term} becomes $$\left(\begin{array}{c} n\\m \end{array} \right)(\sin^2\rho)^{n-m}(\cos^2\rho)^mN^m\int_Mc_1^m(L)$$
By the asymptotic Riemann-Roch formula (ref. Theorem 1.1.22\cite{L})for $N$ big enough
$$n=\frac{\int_Mc_1^m(L)}{m!}N^m+O(N^{m-1})$$
So the leading term is about
$$n^{m+1}(\sin^2\rho)^{n-m}(\cos^2\rho)^m$$

Therefore we have the following theorem
\begin{theo}
With the notations above, for $m\geq 1$ and for $N$ big enough
$$E\chi(\frac{s(z)}{\sqrt{\Pi_N(z,z)}}>\cos \rho)=(1+o(1))n^{m+1}(\sin^2\rho)^{n-m}(\cos^2\rho)^m$$
\end{theo}
\smallskip

Again, plugging in $u=\cos\rho$, we get theorem \ref{theo:estimation}
\smallskip

Let $r_n=\sup\{\rho_n$ that satisfies the requirement of theorem  \ref{theo:eec}$\}$. we are going to show that although  $r_n$ might get
smaller as $n$ grows, there is a positive lower bound.

\section{Critical Radius}\label{sec:critical radius}
We will first analyze the case of Riemann surfaces, then generalize the results to that of higher dimensional smooth projective variety.

First we talk about a little geometry of $\CP^n$ with the Fubini-Study metric. For any $p=[c_0,c_1,\cdots,c_n]\in \CP^n$ the points $q$ such that $d(p,q)=\pi/2$  form a hyperplane $H(p)$ defined by $\sum_0^n c_i\bar{Z_i}=0$ and for any $q\in H$ the complex line that connects $p$ and $q$ is orthogonal to $H(p)$ at $q$. We call $H$ the orthogonal hyperplane of $q$. Conversely, any complex line that are orthogonal to $H(p)$ at a point $q\in H(p)$ must goes through $p$. Also any linear subspace is geodesic.

\subsection{Riemann Surfaces}
Let $X$ be a compact Riemann Surface, and let $(L,h)$ be a positive
Hermitian holomorphic line bundle over $X$. The curvature of $(L,h)$
induces a K\"{a}hler metric on $X$ with K\"{a}ler form
$\omega=\frac{i}{2}\Theta_h$. Let $s_0,s_1,\cdots,s_n$ be an
orthonormal basis of $H^0(X,L^N)$. Here we write $n$ instead of
$n(N)$ for short. This gives an embedding $\Phi_N:X\rightarrow
\CP^n$ for $N$ big enough by Kodaira. If we choose a holomorphic
local frame $e_L$of $L$, then $s_i=f_ie_L^N$ with $f_i$ holomorphic
functions. So $\Phi_N$ is locally given by
$\Phi_N(z)=[f_0(z),f_1(z),\cdots,f_n(z)]$. We denote the vector
$(f_0(z),f_1(z),\cdots,f_n(z))\in \C^{n+1}$ by $F(z)$, and the
vector $(f'_0(z),f'_1(z),\cdots,f'_n(z))$ by $F'(z)$. At each point
$\Phi_N(z)$, the holomorphic tangent line is given by
$[F(z)+tF'(z)],t\in \C\cup \{{\infty}\}$. By $[v]$ for $v\in
\C^{n+1}$, we mean the image under the projection
$\pi:\C^{n+1}\rightarrow \CP^n$. Consider the normal bundle $N\subset T\CP^n|_{\Phi_N(X)}$. At any point $p\in \Phi_N(X)$, $
\exp(N_p)$ is the hyperplane $H_p$ passing $p$ which is orthogonal to $T_p\Phi_N(X)$ at $p$. We define $T_{\infty}(z)$ as the
only point on the tangent line through $\Phi_N(z)$ with distance
$\pi/2$ to $\Phi_N(z)$ in $\CP^N$. Then
$$[T_{\infty}(z)]=[F'(z)-\frac{<F'(z),F(z)>}{<F(z),F(z)>}F(z)]$$
 We
denote by $O_z()$ the projection of $\CP^n$ from $T_{\infty}(z)$ to
its orthogonal hyperplane, which is just $H_z$. In particular we have
$$[O_z([v])]=[v-\frac{<v,T_{\infty}(z)>}{|T_{\infty}(z)|^2}T_{\infty}(z)]$$
Also by $d_N(,)$ we mean the distance in
$\CP^n$ induced by the Fubini-Study metric.

\begin{lem}Let $H_{z\cap w}$ denote the intersection of the normal hyperplanes   $H_z, H_w$ of $\Phi_N(X)$ through $\Phi_N(z)$ and $\Phi_N(w)$ respectively, then
$$\sin^2(d_N(\Phi_N(z),H_{z\cap w}))=\cos^2(d_N(\Phi_N(z),O_z(T_{\infty}(w))))$$
\end{lem}
\begin{proof}
By unitary change of coordinates, we can assume that
$\Phi_N(z)=[0,\cdots,0,1]$,and that
$T_{\infty}(z)=[0,\cdots,0,1,0]$. For any $q\in H_{z\cap w}$, let
$q=[v_0,v_1,\cdots,v_n]$ with $\sum_{i=0}^n |v_i|^2=1$. Then
$\cos(d_N(\Phi_N(z),q))=|v_n|$. So $\cos^2(d_N(\Phi_N(z),H_{z\cap
w}))=max|v_n|^2$. Let $T_{\infty}(w)=[c_0,c_1,\cdots,c_n]$. So the
$v_i$'s satisfies the following equations $$v_{n-1}=0,\ \
\sum_{i=0}^nc_i\bar{v_i}=0$$. So the maximum $|v_n|$ is
$$|v_n|^2=1-\frac{|c_n|^2}{\sum_{i\neq n-1}|c_i|^2}$$ On the other
hand it is clear that
$$\cos^2(d_N(\Phi_N(z),O_z(T_{\infty}(w))))=\frac{|c_n|^2}{\sum_{i\neq
n-1}|c_i|^2}$$ Combining the equations, we get the conclusion.
\end{proof}
Therefore, by switching $z$ and $w$, we have the following equation
$$\sin^2(d_N(\Phi_N(w),H_{z\cap
w}))=\frac{|<F(w),O_w(T_{\infty}(z))>|^2}{|F(w)|^2|O_w(T_{\infty}(z))|^2}$$
where by abuse of notation, we consider the homogeneous coordinate
of a point in $\CP^n $ as a vector in $\C^{n+1}$.

Before we go on calculating the right side of the equation, we
recall the normalized Szeg\"{o} kernels in \cite{sz65} is defined as
$$P_N(z,w)
:=\frac{|\Pi_N(z,w)|}{\Pi_N(z,z)^{1/2}\Pi_N(w,w)^{1/2}}$$
Since $|s_i|_h^2=|f_i|^2h^N$, we have
$$P_N(z,w)=\frac{|<F(z),F(w)>|}{<F(z),F(z)>^{1/2}<F(w),F(w)>^{1/2}}$$
Now we let $E(z,w)=P_N^2(z,w)$, then
$$E(z,w)=\frac{<F(z),F(w)><F(w),F(z)>}{<F(z),F(z)><F(w),F(w)>}$$
Therefore $$\frac{\partial}{\partial
z}E(z,w)=\frac{<F(w),F(z)>}{|F(z)|^2|F(w)|^2}[<F'(z),F(w)>-\frac{<F(z),F(w)>}{|F(z)|^2}<F'(z),F(z)>]$$
and

From now on we use the following convention, by $Df(*,*)$, where $D$
is some differential, we always mean the the value of $Df$ at
$(*,*)$

\begin{eqnarray}
\frac{\partial^2}{\partial z\partial
\bar{w}}E(z,w)&=&\frac{<F(w),F(z)>}{|F(z)|^2|F(w)|^2}[<F'(z),F'(w)>\\&-&\frac{<F(w),F'(w)}{|F(w)|^2}<F'(z),F(w)>
-\frac{<F(z),F'(w)}{|F(z)|^2}<F'(z),F(z)>\\&+&\frac{<F(z),F(w)><F(w),F'(w)>}{|F(z)|^2|F(w)|^2}<F'(z),F(z)>]\end{eqnarray}
We denote $\frac{\partial}{\partial z}E(z,w)$ and
$\frac{\partial^2}{\partial z\partial \bar{w}}E(z,w)$ considered as
functions of $(z,w)$ by $E_z(z,w)$ and $E_{z\bar{w}}(z,w)$
respectively. So in particular
$$E_{z\bar{w}}(z,z)=\frac{1}{|F(z)|^2}[<F'(z),F'(z)>-\frac{|<F'(z),F(z)>|^2}{|F(z)|^2}]$$

Now we calculate $\sin^2(d_N(\Phi_N(w),H_{z\cap w}))$.

First we have
$$O_w(T_{\infty}(z))=T_{\infty}(z)-\frac{<T_{\infty}(z),T_{\infty}(w)>}{|T_{\infty}(w)|^2}T_{\infty}(w)$$ since $O_w(T_{\infty}(z))$ is orthogonal to $T_{\infty}(w)$ we have
$$|O_w(T_{\infty}(z))|^2=|T_{\infty}(z)|^2-\frac{|<T_{\infty}(z),T_{\infty}(w)>|^2}{|T_{\infty}(w)|^2}$$
and since $F(w)$ is also orthogonal to $T_{\infty}(w)$ we also have
$$|<F(w),O_w(T_{\infty}(z))>|^2=|<F(w),T_{\infty}(z)>|^2$$

Since
$$<T_{\infty}(z),F(w)>=<F'(z),F(w)>-\frac{<F'(z),F(z)>}{|F(z)|^2}<F(z),F(w)>$$
We get the following equation
\begin{lem}
With the notations above
$$E_z(z,w)=\frac{<F(w),F(z)>}{|F(z)|^2|F(w)|^2}<T_{\infty}(z),F(w)>$$
\end{lem}
\smallskip
Moreover,
$$|T_{\infty}(z)|^2=<F'(z),F'(z)>-\frac{|<F'(z),F(z)>|^2}{|F(z)|^2}$$
and
\begin{eqnarray}<T_{\infty}(z),T_{\infty}(w)>&=&<F'(z),F'(w)>-\frac{<F(w),F'(w)}{|F(w)|^2}<F'(z),F(w)>
\\&-&\frac{<F(z),F'(w)}{|F(z)|^2}<F'(z),F(z)>\\&+&\frac{<F(z),F(w)><F(w),F'(w)>}{|F(z)|^2|F(w)|^2}<F'(z),F(z)>
\end{eqnarray}
Therefore we have
$$E_{z\bar{w}}(z,z)=\frac{1}{|F(z)|^2}|T_{\infty}(z)|^2$$ and
$$E_{z\bar{w}}(z,w)=\frac{<F(w),F(z)>}{|F(z)|^2|F(w)|^2}<T_{\infty}(z),T_{\infty}(w)>$$

Combining these equations we have
\begin{eqnarray}\sin^2(d_N(\Phi_N(w),H_{z\cap w})) \qquad \qquad \qquad \qquad \qquad \qquad\qquad \qquad \qquad \qquad \qquad \qquad  \\=\frac{(|F(z)|^2|F(w)|^2)^2|E_z(z,w)|^2}{|<F(w),F(z)>|^2|F(w)|^2\{|F(z)|^2E_{z\bar{w}}(z,z)-\frac{[|F(z)|^2|F(w)|^2]^2|E_{z\bar{w}(z,w)}|^2}{|F(w)|^2E_{z\bar{w}}(w,w)|<F(w),F(z)>|^2}\}}\\=\frac{|E_{z}(z,w)|^2}{E(z,w)[E_{z\bar{w}}(z,z)-\frac{|E_{z\bar{w}}(z,w)|^2}{E(z,w)E_{z\bar{w}}(w,w)}]}\qquad \qquad \qquad \qquad \qquad \qquad \qquad \qquad \ \
\end{eqnarray}

So we have the following theorem
\begin{theo}\label{final}
With the notations above we have the equation
$$\sin^2(d_N(\Phi_N(w),H_{z\cap w}))=\frac{|E_{z}(z,w)|^2}{E(z,w)E_{z\bar{w}}(z,z)[1-\frac{|E_{z\bar{w}}(z,w)|^2}{E_{z\bar{w}}(z,z)E_{z\bar{w}}(w,w)}\frac{1}{E(z,w)}]}$$
\end{theo}
\smallskip

As in the last section, we choose local coordinates such $z_0=0$.
Then $$E(z,w)=P_N^2(z,w)=e^{-|u-v|^2}[1+R_N(u,v)]^2,$$ where
$u=\sqrt{N}z,v=\sqrt{N}w$. So
 \begin{eqnarray}
\frac{\partial}{\partial z}E(z,w)&=&\sqrt{N}\frac{\partial}{\partial
u}e^{-|u-v|^2}[1+R_N(u,v)]^2\\
&=&\sqrt{N}[e^{-|u-v|^2}(\bar{v}-\bar{u})[1+R_N(u,v)]^2\\&+&e^{-|u-v|^2}2(1+R_N(u,v))\frac{\partial}{\partial
u}R_N(u,v)]
\end{eqnarray}
When $z=0$
$$E(0,w)=e^{-N|w|^2}[1+O(N^{1/2+\epsilon}|w|^2)]^2$$
\begin{eqnarray}
\frac{\partial}{\partial
z}E(0,w)&=&\sqrt{N}e^{-N|w|^2}(1+O(N^{-1/2+\epsilon})\sqrt{N}\bar{w}+O(N^{\epsilon}|w|))\\&=&(1+o(1))Ne^{-N|w|^2}\bar{w}
\end{eqnarray}

where we make $\epsilon<1/2$ and use the estimation that
$$|\frac{\partial }{\partial u}R_N(u,v)_{u=0}|\leq
C_2|v|N^{-1/2+\epsilon}$$

Furthermore we can calculate
\begin{eqnarray}
\frac{\partial^2}{\partial z\partial
\bar{w}}E(z,w)&=&N\frac{\partial}{\partial
\bar{v}}\{\sqrt{N}[e^{-|u-v|^2}(\bar{v}-\bar{u})][1+R_N(u,v)]^2\\&+&e^{-|u-v|^2}2(1+R_N(u,v))\frac{\partial}{\partial
u}R_N(u,v)\}\\&=&Ne^{-|u-v|^2}(u-v)[(\bar{v}-\bar{u})(1+R_N)^2)+2(1+R_N)\frac{\partial}{\partial
u}R_N]\\&+&Ne^{-|u-v|^2}\{(1+R_N)^2+(\bar{v}-\bar{u})\frac{\partial}{\partial
\bar{v}}(1+R_N)^2\\&+&2\frac{\partial}{\partial
\bar{v}}R_N\frac{\partial}{\partial
u}R_N+2(1+R_N)\frac{\partial^2}{\partial u \partial \bar{v}}R_N\}
\end{eqnarray}
Therefore
\begin{eqnarray}
E_{z\bar{w}}(w,w)&=&N[(1+R_N)^2+2\frac{\partial}{\partial
\bar{v}}R_N\frac{\partial}{\partial
u}R_N](v,v)\\&=&N(1+2\frac{\partial^2}{\partial u\partial
\bar{v}}R_N(v,v))
\end{eqnarray}
and

\begin{eqnarray}
E_{z\bar{w}}(0,w)&=&Ne^{-N|w|^2}\{(1+O(N^{1/2+\epsilon}|w|^2))^2(1-N|w|^2)\\&+&2(1+O(N^{1/2+\epsilon}|w|^2))O(N^{1/2+\epsilon}|w|^2)\\&+&O(N^{1/2+\epsilon}|w|^2)(1+O(N^{1/2+\epsilon}|w|^2))\\&+&O(N^{2\epsilon}|w|^2)+2(1+O(N^{1/2+\epsilon}|w|^2)){\partial
u\partial
\bar{v}}R_N(0,v))\}\\&=&Ne^{-N|w|^2}[1-(1+o(1))N|w|^2+2\frac{\partial^2}{\partial
u\partial \bar{v}}R_N(0,v)]
\end{eqnarray}
With the notations above we have the following theorem

\begin{theo}\label{theo:local}
There exist $r(L,h)>0$, which is independent of $\ \ N$, such that
$$\sin^2(d_N(\Phi_N(w),H_{0\cap w}))>r(L,h)$$ for
$|w|<\frac{1}{\sqrt{2N}}$
\end{theo}

\begin{proof}
Before applying theorem  \ref{final}, we need the following estimations
\begin{eqnarray}
\frac{\partial^2}{\partial u\partial
\bar{v}}R_N(0,v)-\frac{\partial^2}{\partial u\partial
\bar{v}}R_N(0,0)=A(v)+O(N^{-1/2+\epsilon}|v|^2)\\\frac{\partial^2}{\partial
u\partial \bar{v}}R_N(v,v)-\frac{\partial}{\partial u\partial
\bar{v}}R_N(0,0)=A(v)+\overline{A(v)}+O(N^{-1/2+\epsilon}|v|^2)
\end{eqnarray}
where $A(v)=O(N^{-1/2+\epsilon}|v|)$ .

Note first that $R_N(u,v)$ is a real analytic function, so we can
write $R_N(u,v)$ as power series in $(u,\bar{u},v,\bar{v})$. So the
first equation follows directly from theorem [\cite{sz65}]. Now we prove
the second equation.

We denote by $g(u,v)$ the homogeneous part of degree 3 in the power
series, since this is the part that contribute terms of degree 1 in
the second derivatives. Notice that $g_{u\bar{v}}(0,0)=0$,we need to
show that
$$g_{u\bar{v}}(0,x)+\overline{g_{u\bar{v}}(0,x)}=g_{u\bar{v}}(x,x)$$
for all $x\in \C$.

Since $g(u,v) $ is real,
$\overline{g_{u\bar{v}}(0,x)}=g_{v\bar{u}}(0,x)$. We write
$p(x)=g_{u\bar{v}}(x,x)-g_{u\bar{v}}(0,x)-g_{v\bar{u}}(0,x)$. So
$p(x)$ is linear in $(x,\bar{x})$ and $p(0)=0$. To show that
$p(x)\equiv 0$, we just need to show that $\frac{\partial}{\partial
x}p(x)=0$ and $\frac{\partial}{\partial \bar{x}}p(x)=0$. But
$$\frac{\partial}{\partial
x}p(x)=g_{uu\bar{v}}(x,x)+g_{vu\bar{v}}(x,x)-g_{vu\bar{v}}(0,x)-g_{vv\bar{u}}(0,x)=0$$
where the second equation follows from the fact that all terms in
the middle are constant and that $g$ is symmetic with respect to $u$
and $v$. So we have proved the second equation.

We let $a=\frac{\partial^2}{\partial u\partial \bar{v}}R_N(0,0)$,
then $a=O(N^{-1/2+\epsilon})$ and $a$ is real. So

\begin{eqnarray}
E_{z\bar{w}}(0,0)&=&N(1+2a)\\
E_{z\bar{w}}(0,w)&=&Ne^{-N|w|^2}[1-(1+o(1))N|w|^2+2a+2A(v)+O(N^{-1/2+\epsilon}|v|^2)]\label{eqn:xyz}\\
E_{z\bar{w}}(w,w)&=&N[1+2a+2A(v)+2\overline{A(v)}+O(N^{-1/2+\epsilon}|v|^2)]
\end{eqnarray}

Plug in these estimations together with the ones about $E_z(0,w)$
and $E(0,w)$ to the expression of $\sin^2(d_N(\Phi_N(w),H_{0\cap
w}))$ in the last theorem, and use the Taylor series of $e^{-N|w|^2}$, both the numerator and denominator is bounded by positive multiples of $N|w|^2$, then it is easy to see that there is a
constant $r>0$ independent of $N$ and $w$ for
$w<\frac{1}{\sqrt{2N}}$ such that $\sin^2(d_N(\Phi_N(w),H_{0\cap
w}))>r$. Also since the constant in the approximation of the normalized Szeg\"{o} kernel is independent of the point $z$, $r$ can be chosen independent of $z$.
\end{proof}

As a corollary, we have the following theorem
\begin{theo}\label{theo:conclusion}
Let $r_N$ be the critical radius of $\Phi_N(X)$ considered as a submanifold of $\CP^{n}$. There exists a constant $c(X,h)>0$, such that $r_N>c(X,h)$
\end{theo}
\begin{proof}
We still use the preferred coordinates chosen centered at $z$

Let $N_x(b)=\{v\in N_x(\Phi_N(X)),\parallel
v\parallel \leq b \}$
Notice that by theorem  \ref{theo:sz1} and  \ref{theo:sz2}, for $w\geq \frac{1}{\sqrt{2N}}$, and for $N$ big enough, $d_N(\Phi_N(z),\Phi_N(z+w))\geq \cos^{-1}[(1+o(1))e^{-1/4}]$.

Combining this fact and theorem  \ref{theo:local}, there exists a constant $c>0$, which is independent of $z$ such that for any point $q\in \Phi_N(X)$, $$\exp_{\Phi_N(z)}(N_{\Phi_N(z)}(c))\cap \exp_q(N_q(c))=\emptyset$$
This implies that the critical radius is bounded below, namely $r_N>c(X,h)$
\end{proof}

\subsection{Higher Dimension}
Actually the argument for Riemann surfaces carries directly to high dimensional K\"{a}hler manifolds. Now use the notations in section  \ref{sec:generalsettings}, we have the following theorem

\begin{theo}
Let $r_N$ be the critical radius of $\Phi_N(M)$ considered as a submanifold of $\CP^{n}$. There exists a constant $c(M,h)>0$, such that $r_N>c(M,h)$
\end{theo}
\begin{proof}
We just need a high dimensional  version of theorem  \ref{theo:local}.

We still choose a preferred coordinates centered at $z$, and let $w<\frac{1}{\sqrt{2N}}$.

In order to apply theorem  \ref{final}, we let $X$ be the complex line in the coordinates chart connecting $0$ and $w$, and restrict $\Phi_N$ to an open set $V\subset X$. Then all the estimations we used in proving theorem  \ref{theo:local} hold for $V$. So theorem  \ref{theo:local} can be applied to $V$. Notice that the lower bounds we can get come from the approximation of the normalized Szeg\"{o} kernel of $(L,h)\rightarrow M$, hence is independent of $w$

 Notice that the normal space of $M$ at $\Phi_N(z)$ is contained in the normal hyperplane of $\Phi_N(V)$ at $\Phi_N(z)$, the same is true for $\Phi_N(z+w)$. Therefore the intersection of the two normal spaces $N_{0\cap w}=N_{\Phi_N(z)}(M)\cap N_{\Phi_N(z+w)}(M) $ of $M$ is contained in the intersection of the two normal hyperplanes. Therefore the distance from $\Phi_N(z)$ to the intersection $N_{0\cap w}$ is also bounded below independent of $z$, $w$ and $N$.

 Now use the same argument as in theorem  \ref{theo:conclusion}, we get the expected conclusion.
\end{proof}

\end{document}